%% file: main.tex
\documentclass[11pt,dvipsnames]{amsart}
\input{preamble}

\usepackage[margin=1in]{geometry}

\author[N. Wisdom]{Noah Wisdom}
\address{Department of Mathematics,
         Northwestern University,
         Evanston, IL, 60208
         USA}
\email{NoahAnkney2026@u.northwestern.edu}



\title{Affine \'{e}tale group schemes over Tambara fields}
\date{}

\begin{document}
\begin{abstract}
    We classify finite \'{e}tale extensions and finite affine \'{e}tale group schemes over the $G$-Tambara functor $\underline{\mathbb{F}}$, for $\mathbb{F}$ any algebraically closed field and $G$ any finite group. This establishes $G$-Galois descent from the Tambara functor algebraic closure of $\underline{\mathbb{F}}$. In particular, we find new families of \'{e}tale extensions of any $G$-Tambara functor and show that, together with one of the families discovered by Lindenstrauss--Richter--Zou, these give all finite \'{e}tale extensions of $\underline{\mathbb{F}}$. Our arguments also show that the map $\underline{K} \rightarrow \mathrm{FP}(L)$ associated to any $G$-Galois extension $L$ of $K$ is \'{e}tale, generalizing a result of Lindenstrauss--Richter--Zou when $G$ is cyclic. Lastly, we classify flat finitely generated $\underline{\mathbb{F}}$-modules when $G = C_p$.
\end{abstract}

\maketitle
\tableofcontents

\input{Introduction}
\input{Background}
\input{Technical_Lemmas}
\input{MainResults}

\bibliographystyle{alpha}
\bibliography{ref}

\end{document}

%% file: preamble.tex
\usepackage{xy}
\usepackage{graphicx, amsmath, amssymb, amsthm,amsfonts}
\usepackage[mathscr]{eucal}
\usepackage[pagebackref, colorlinks, citecolor=Red, linkcolor=NavyBlue, urlcolor=NavyBlue]{hyperref}
\hypersetup{linktoc=all}
\usepackage{comment}
\usepackage{hyperref}
\usepackage[capitalise, nameinlink]{cleveref}
\usepackage{microtype}
\usepackage{tikz-cd}
\usepackage{spectralsequences}

\newcommand{\Q}{\mathbb{Q}}
\newcommand{\Z}{\mathbb{Z}}
\newcommand{\A}{\mathcal{A}}
\newcommand{\F}{\mathbb{F}}

\crefname{lemma}{Lemma}{Lemmas}
\crefname{theorem}{Theorem}{Theorems}
\crefname{definition}{Definition}{Definitions}
\crefname{proposition}{Proposition}{Propositions}
\crefname{remark}{Remark}{Remarks}
\crefname{corollary}{Corollary}{Corollaries}
\crefname{equation}{Equation}{Equations}
\crefname{construction}{Construction}{Constructions}
\crefname{ex}{Example}{Examples}
\crefname{example}{Example}{Example}
\crefname{appsec}{Appendix}{Appendices}
\crefname{subsection}{Subsection}{Subsections}

\newtheorem{theorem}{Theorem}[section]
\newtheorem{letterthm}{Theorem}

\newtheorem{lettercor}[letterthm]{Corollary}
\newtheorem{lemma}[theorem]{Lemma}

\newtheorem{proposition}[theorem]{Proposition}
\newtheorem{corollary}[theorem]{Corollary}

\newtheoremstyle{BoldRemark} 
{10pt}                    
{10pt}                    
{\upshape}                   
{}                           
{\bfseries}                  
{.}                          
{.5em}                       
{}  
\theoremstyle{BoldRemark}
\newtheorem{remark}[theorem]{Remark}

\newtheorem{definition}[theorem]{Definition}
\newtheorem{letterdef}[letterthm]{Definition}

\usepackage[textwidth=25mm, textsize=tiny]{todonotes}

\newcommand{\Res}{\mathrm{Res}}
\newcommand{\res}{\mathrm{res}}

\newcommand{\Nm}{\mathrm{Nm}}

\newcommand{\Coind}{\mathrm{CoInd}}
\newcommand{\Hom}{\mathrm{Hom}}
\newcommand{\colim}{\mathrm{Colim}}
\newcommand{\Der}{\mathrm{Der}}
\newcommand{\Tr}{\mathrm{Tr}}

\newcommand{\FP}{\mathrm{FP}}
\newcommand{\ev}{\mathrm{ev}}
\newcommand{\fet}{\mathsf{Grp}^{\mathrm{f\acute{e}t\text{-}aff}}}

\newcommand{\Mod}{\mathrm{Mod}}
\newcommand{\Alg}{\mathrm{Alg}}
\newcommand{\Ring}{\mathrm{Ring}}

\newcommand{\Tamb}{\mathrm{Tamb}}
\newcommand{\Green}{\mathrm{Green}}

%% file: Introduction.tex
\section{Introduction}

Tambara functors are equivariant analogues of rings arising in representation theory, group cohomology, and equivariant homotopy theory. Recently, the commutative algebra of Tambara functors has been studied by homotopy theorists motivated in part by the hope of establishing equivariant analogues of classical applications of homological algebra and algebraic geometry to equivariant homotopy theory.

For example, \cite{4DS} study the Nakaoka spectrum (the equivariant analogue of the Zariski spectrum) of some free polynomial Tambara functors motivated in part by the hope of gaining insight into the algebraic geometry of the Tambara affine line. On the other hand, \cite{HMQ23} observe that free polynomial Tambara functors often fail to be flat as modules, suggesting that new ideas may be necessary to approach some of the homological algebra which appears in equivariant stable homotopy theory.

One important concept in algebraic geometry is the notion of \'{e}taleness. For example, Grothendieck introduced \'{e}tale cohomology and used it to prove three of the four Weil conjectures. In fact, Grothendieck was partially motivated by the observation that the Zariski topology was not always the correct topology. Equivariant algebra has the same issue: there exists a basic open subset of the Nakaoka spectrum of a Tambara functor which is not homeomorphic to the Nakaoka spectrum of any Tambara functor (in particular, there is no localization producing the correct Nakaoka spectrum).

In \cite{Hil17}, Hill introduced genuine K\"{a}hler differentials in the Tambara functor setting and showed that the genuine K\"{a}hler differentials support the universal Tambara functor derivation. Hill also proposed defining a \emph{formally \'{e}tale} Tambara functor morphism as a map of Tambara functors which is both flat and for which the genuine K\"{a}hler differentials vanish. A Tambara functor morphism is \emph{\'{e}tale} if it is finitely presented and formally \'{e}tale. Later in \cite{LRZ24} Lindenstrauss, Richter, and Zou discovered interesting families of (formally) \'{e}tale morphisms of $C_n$-Tambara functors.

To start with, we establish many familiar properties of \'{e}tale morphisms: they are closed under composition, passing to projections, taking products and coinduction (a twisted product with no analogue in the nonequivariant world), and they are preserved by base-change. The principal difficulty here is in checking the finitely presented condition, as free polynomial Tambara functors are not known to always satisfy the ascending chain condition on ideals. Next, we find more examples of \'{e}tale extensions.

\begin{letterthm}(cf. \cref{thm:coind-are-etale})
    For any $G$-Tambara functor $k$ and subgroup $H \subset G$, the canonical map 
    \[ 
        k \rightarrow \Coind_H^G \Res_H^G k 
    \] 
    is \'{e}tale.
\end{letterthm}

\begin{remark}
    In \cite{Wis25b}, the author gives explicit computations of the map of Nakaoka spectra induced by the \'{e}tale map $k \rightarrow \Coind_H^G \Res_H^G k$. Surprisingly, when $k$ is the Burnside Tambara functor, the induced map on Nakaoka spectra is closed and not open (when $H \neq G$). This contrasts the classical fact that any \'{e}tale map of schemes is open.
\end{remark}

Previous work of Schuchardt, Spitz, and the author shows that the Tambara functors $\Coind_e^G \F$, for $\F$ an algebraically closed field, should be viewed as the Tambara functor analogues of algebraically closed fields. Now $\Coind_e^G \F$ carries a natural $G$-action, whose fixed-points are the constant Tambara functor $\underline{\F}$. Already we are encountering interesting new behavior in the equivariant algebra world: classically, the only finite group which acts faithfully on an algebraically closed field is $C_2$, in which case we must be in characteristic zero \cite{AS27a, AS27b}. 

As observed in \cite{SSW24}, $\Coind_e^G$ gives a faithful embedding of the category of rings in the category of $G$-Tambara functors, so that the finite \'{e}tale $R$-algebras correspond precisely to finite \'{e}tale $\Coind_e^G R$-algebras. In particular, an arbitrary $\Coind_e^G \F$-algebra $\Coind_e^G S$ is \'{e}tale if and only if $S$ is an \'{e}tale $\F$-algebra, for $\F$ an algebraically closed field.

Guided by the hope that Galois descent along $\underline{\F} \rightarrow \Coind_e^G \F$ holds, we study finite affine \'{e}tale group schemes, defined as follows.

\begin{letterdef}
    A finite affine \'{e}tale group scheme over a ground Tambara functor $k$ is a representable functor 
    \[ 
        \Hom_{k\text{-}\Alg}(R,-) : k\text{-}\Alg \rightarrow \mathsf{Grp} 
    \] 
    such that $k \rightarrow R$ is finite and \'{e}tale. We will write 
    \[ 
        \fet_k 
    \] 
    for the category of such objects (with morphisms the natural transformations).

    A finite affine \'{e}tale group scheme is commutative if it takes values in abelian groups.
\end{letterdef}

\begin{remark}
    The finite presentation assumption in the definition of \'{e}tale implies that any finite affine \'{e}tale group scheme preserves filtered colimits.
\end{remark}

With sufficient control over flat modules, we are able to completely classify finite \'{e}tale $G$-Tambara functor maps out of the constant $G$-Tambara functor $\underline{\F}$, for $\F$ an algebraically closed field. This is the first result in the literature which classifies all \'{e}tale extensions of a given Tambara functor.

\begin{letterthm}(cf. \cref{thm:modular-characteristic-characterizing-finite-etale})
	Let $G$ be an arbitrary finite group and $\F$ any algebraically closed field. A finite $\underline{\F}$-algebra is \'{e}tale if and only if it is a finite product of \'{e}tale $\underline{\F}$-algebras 
    \[ 
    \underline{\F} \rightarrow \Coind_H^G \underline{\F} . 
    \]
\end{letterthm}

Recall the topologist's notation $\mathcal{C}^{BG}$ for the category whose objects are the objects of a category $\mathcal{C}$ equipped with an action of the group $G$, and whose morphisms are $G$-equivariant morphisms in $\mathcal{C}$. We are able to establish $G$-Galois descent along $\underline{\F} \rightarrow \Coind_e^G \F$ for finite affine \'{e}tale group schemes.

\begin{lettercor}(cf. \cref{cor:finite-etale-group-schemes-modular-characteristic})
    Let $G$ be an arbitrary finite group and $\F$ any algebraically closed field. Then
    \[ 
        \ev_{G/e} : \fet_{\underline{\F}} \rightarrow (\fet_\F)^{BG} 
    \] 
    is an equivalence of categories with inverse induced by the fixed-point construction $\FP$. The subcategory of commutative group schemes is equivalent to the category of finitely generated $\Z[G]$-modules.
\end{lettercor}

\begin{remark}
    From a homotopy theorist's point of view, $\Coind_e^G \F$ is very uninteresting (for example, Bredon cohomology with $\Coind_e^G R$ coefficients only detects the underlying space of a $G$-space). On the other hand, $\underline{\F}$ is an extremely common choice of coefficients---especially $\underline{\F_2}$ when $G = C_2$ (cf. \cite{BW18,DHM24,HW20,Haz21,May20,Pet24}).
    
    Galois descent in chromatic homotopy theory is an important tool, for example along the Galois extension $L_{K(n)} \mathbb{S} \rightarrow E_n$ with profinite Galois group the Morava stabilizer group. The topological analogue of Galois descent along $R \rightarrow \Coind_e^G R$ is the homotopy fixed point spectral sequence, which is of fundamental importance. For example, this spectral sequence relates $\mathrm{THH}$ and $\mathrm{TC}^-$.
\end{remark}

When $|G|$ is invertible we can treat more general base Tambara functors.

\begin{letterthm}(cf. \cref{thm:detecting-etale-at-bottom-level})
	Let $\ell$ be a flat finitely presented $k$-algebra in $G$-Tambara functors and assume either
	\begin{enumerate}
		\item $\ell$ is cohomological and $|G|$ is invertible in $\ell(G/G)$, or
		\item all transfers in $\ell$ are surjective.
	\end{enumerate}
	Then $\ell$ is \'{e}tale over $k$ if and only if $\ell(G/e)$ is \'{e}tale over $k(G/e)$.
\end{letterthm}

Using \cref{thm:detecting-etale-at-bottom-level} we may generalize \cite[Theorem 4.4]{LRZ24}. Specifically, \cite[Theorem 4.4]{LRZ24} is the special case of the following result in which $K$ is a $C_n$-Kummer extension of a field $L$.

\begin{lettercor}(cf. \cref{cor:Galois-field-extension-gives-etale})
    Let $L$ be a $G$-Galois extension of a field $K$. Then
    \[ 
        \underline{K} \rightarrow \FP(L) 
    \] 
    is formally \'{e}tale. Under a mild technical condition ($\underline{K}$ satisfies the Hilbert basis theorem) it is \'{e}tale.
\end{lettercor}

In a different direction, \cref{thm:detecting-etale-at-bottom-level} allows us to establish $G$-Galois descent for affine \'{e}tale group schemes along $\underline{R} \rightarrow \Coind_e^G R$.

\begin{lettercor}(cf. \cref{cor:Theorem-Q})
    Let $R$ be a ring such that $|G|$ is a unit in $R$. Under a mild technical condition ($\underline{R}$ satisfies the Hilbert basis theorem), the functor 
    \[ 
        \ev_{G/e} : \fet_{\underline{R}} \rightarrow (\fet_R)^{BG} 
    \] 
    is an equivalence of categories with inverse induced by $\FP$.
\end{lettercor}

For example, if $\A_G$ satisfies the Hilbert basis theorem and $R$ is Noetherian, then $\underline{R}$ satisfies the Hilbert basis theorem \cite{Sun25}. $\A_G$ is known to satisfy the Hilbert basis theorem for many groups: $C_p$ \cite{4DS}, Dedekind groups \cite{Sun25}, and more (see \cite{Sun25} for a more complete list).

On the other hand, in modular characteristic and when $G = C_p$ we obtain the following. It is of independent interest, and it may also be useful in the study of \'{e}tale $\underline{\F}$-algebras when $\F$ has characteristic $p$.

\begin{letterthm}(cf. \cref{thm:flat-implies-free-for-constant-F_p})
    Let $\F$ be any field of characteristic $p$ and $G = C_p$. The following conditions are equivalent for a finitely generated $\underline{\F}$-module $M$: 
    \begin{enumerate}
        \item $M$ is flat
        \item $M$ is projective
        \item $M$ is free.
    \end{enumerate}
\end{letterthm}

Besides flatness results, the other two ingredients in the proofs of \cref{thm:modular-characteristic-characterizing-finite-etale,cor:finite-etale-group-schemes-modular-characteristic} are the classification of finite \'{e}tale extensions of algebraically closed fields, and \cite[Propositions 3.6 and 3.8]{Wis25a}, which allow one to deduce some of the structure of a Tambara functor $k$ from its bottom level $k(G/e)$. As the results of \cite{Wis25a} that we use are false for Green functors, we do not expect any variant of our argument to produce similar theorems for Green functors in modular characteristic.

\begin{remark}
    In \cite{Mat17} Mathew establishes a base-change formula for topological Hochschild homology along \'{e}tale maps of $\mathbb{E}_\infty$ ring spectra. An equivariant analogue of this result would be desirable in the emerging picture of equivariant trace methods, as equivariant $\mathrm{THH}$ (such as that considered in \cite{CGK25}) is difficult to compute. Note that \cite[Theorem 4.19]{MQS24} establishes a connection between Hochschild homology and the K\"{a}hler differentials used to define \'{e}taleness in the equivariant algebra setting.
    
    A map of commutative ring spectra is \'{e}tale if two conditions hold; one of them is that the induced map on $\pi_0$ is an \'{e}tale map of ordinary rings. A natural $G$-equivariant generalization of this property in the situation of a map of $G$-$\mathbb{E}_\infty$ ring $G$-spectra is to ask for the induced map on $\underline{\pi}_0$ to be an \'{e}tale map of Tambara functors. The author proposes that there is a $G$-equivariant analogue of Mathew's base-change formula which involves such a hypothesis; it seems like the main obstruction in establishing such a result is ensuring that an equivariant analogue of \cite[Proposition 7.5.1.15]{Lur17} goes through.
\end{remark}

\vspace{3mm}

\textbf{Acknowledgments.} The author thanks Mike Hill for a wealth of insightful suggestions, and for explaining the proof of \cref{thm:Green-etale-implies-Tambara-etale}. Additionally, the author thanks David Chan, Ayelet Lindenstrauss, Jackson Morris, and Birgit Richter for helpful conversations. Lastly, the author thanks Yuchen Liu for teaching a wonderful course in algebraic geometry that was very inspirational.

%% file: Background.tex
\section{Recollections on the commutative algebra of Tambara functors}

In this section we collect some known results in equivariant algebra and review the definition of \'{e}taleness. We assume the reader is familiar with Mackey, Green, and Tambara functors. A module over a Tambara functor is a module over the underlying Green functor.

Throughout, $G$ is a finite group. We start by collecting some results on modules.

\begin{lemma}\label{lem:sufficient-conditions-for-all-modules-to-be-FP}
	Let $\ell$ be a Tambara functor such that either
	\begin{enumerate}
		\item $\ell$ is cohomological and $|G|$ is invertible in $\ell(G/G)$, or
		\item all transfers in $\ell$ are surjective
	\end{enumerate}
	then all restrictions in every $\ell$-module are injective.
\end{lemma}

\begin{proof}
	If $|G|$ is invertible in $\ell(G/G)$, then it is invertible in each ring $\ell(G/H)$. The cohomological assumption then implies that all transfers are surjective, so the first hypothesis is a special case of the second hypothesis.
	
	Assume all transfers in $\ell$ are surjective and let $H \subset K$ be arbitrary. Then $1 \in \ell(G/K)$ is the transfer of some $y \in \ell(G/H)$. If $M$ is an $\ell$-module, Frobenius reciprocity implies that the restriction $\Res_H^K$ in $M$ is injective: 
    \[ 
        m = 1 \cdot m = \Tr_H^K(y) \cdot m = \Tr_H^K(y \cdot \Res_H^K(m)) .
    \]
\end{proof}

\begin{definition}[\cite{CW25}]
    A Tambara functor $k$ is \emph{relatively finite dimensional} if each restriction $\Res_H^G$ is a finite ring map ($k(G/H)$ is a finitely generated $k(G/G)$-module).
\end{definition}

\begin{definition}
	We say that a map $k \rightarrow \ell$ of Tambara functors is \emph{finite} if $\ell$ is a finitely generated $k$-module.
\end{definition}

If $k$ is relatively finite dimensional, then by \cite[Proposition 3.31]{CW25} $k \rightarrow \ell$ is finite if and only if it is levelwise finite.

\begin{definition}
    Let $k$ be a Tambara or Green functor and $M$ a $k$-module. We say $M$ is flat if the functor $M \boxtimes_k -$ is exact.
\end{definition}

If $X$ is a finite $G$-set, We will use the notation $\ev_{X}$ to describe either the functor $k \mapsto k(X)$ from Tambara functors to rings, or from Tambara functors to rings with an action of the Weyl group $\mathrm{Aut}(X)$. It will be clear from context which of these two we mean.

\begin{definition}
    Let $k$ be a Tambara functor. The free polynomial $k$-algebra on a finite $G$-set $X$ is the representing object of the functor 
    \[ 
        \ev_X : k \text{-} \Alg \rightarrow \Ring . 
    \]
    We denote this $k$-algebra by $k[y_X]$ and say $y_X$ is a generator in level $X$.
\end{definition}

\begin{definition}
    We say that a $k$-algebra $R$ is \emph{finitely presented} if $R$ is isomorphic to a coequalizer
    \[ 
        \mathrm{Coeq}(k[x_i]_{i \in I} \rightrightarrows k[x_j]_{j \in J}) 
    \]
    of free $k$-algebras on finitely many generators.
\end{definition}

Equivalently (by \cite[Proposition 4.3]{SSW24}) $R$ is finitely presented if and only if it is \emph{compact} in the category of $k$-algebras, i.e. whenever $F : \mathsf{D} \rightarrow k \text{-} \Alg$ is a filtered diagram in $k$-algebras, the canonical map
\[ 
    \Hom_k(R, \colim_{d \in \mathsf{D}} F(d)) \rightarrow \colim_{d \in \mathsf{D}} \Hom_k(R,F(d))
\] 
is a bijection of sets. From this description it follows that free polynomial $k$-algebras are finitely presented, and any finite colimit of finitely presented $k$-algebras is finitely presented.

\begin{definition}
    We say that a Tambara functor $k$ satisfies the \emph{Hilbert basis theorem} if every ideal of every free polynomial $k$-algebra on finitely many generators is finitely generated (equivalently, every free finitely generated polynomial $k$-algebra satisfies the ascending chain condition on ideals, i.e. is Noetherian).
\end{definition}

\cite[Corollary 3.12]{4DS} and \cite[Theorems A and C]{Sun25} establish fairly general sufficient criteria for a Tambara functor to satisfy the Hilbert basis theorem. However, there exists finite groups $G$ such that it is not currently known whether or not the Burnside Tambara functor $\A_G$ satisfies the Hilbert basis theorem.

If a $k$-algebra $R$ receives a surjection from a free polynomial $k$-algebra on finitely many generators, we say that $R$ is finitely generated over $k$. Finitely presented implies finitely generated; the converse is true if and only if $k$ satisfies the Hilbert basis theorem.

Next, we review of the definition and first properties of \'{e}tale morphisms of Tambara functors, following the original definition due to Hill \cite{Hil17} and the treatment by Lindenstrauss, Richter, and Zou \cite{LRZ24}. We start by working towards the definition of genuine K\"{a}hler differentials of a morphism of Tambara functors.

\begin{definition}[\cite{Hil17}]
	Let $I$ be a Tambara ideal of a Tambara functor $R$. Define $I^{>1}$ to be the sub-ideal of $I$ generated by all nontrivial norms (including products) of elements of $I$. Explicitly, $I^{>1}(G/H)$ is generated by $I(G/H)^2$ along with the images of all norms $I(G/K) \rightarrow I(G/H)$ for $K$ conjugate to a proper subgroup of $H$.
\end{definition}

\begin{definition}[\cite{Hil17}]
	Let $k \rightarrow R$ be a morphism of Tambara functors, let $I$ be the kernel of the map 
    \[ 
        R \boxtimes_k R \rightarrow R \mathrm{,} 
    \] 
    and define the genuine K\"{a}hler differentials by 
    \[ 
        \Omega^{1,\Tamb}_{R/k} := I/I^{>1} . 
    \]
\end{definition}

\begin{definition}[\cite{Hil17}]
	A morphism $k \rightarrow R$ of Tambara functors is \emph{formally \'{e}tale} if $\Omega^{1,\Tamb}_{R/k} = 0$ and $R$ is flat as a $k$-module.
\end{definition}

As in the classical case, the genuine K\"{a}hler differentials support the universal derivation.

\begin{definition}[{\cite[Definition 4.1]{Hil17}}]\label{def:genuine-k-derivation}
    Let $R$ be a $k$-algebra and $M$ an $R$-module. A $k$-module morphism $d : R \rightarrow M$ is a \emph{genuine $k$-derivation} if
    \begin{enumerate}
        \item for all $H \subset G$, $r_1, r_2 \in R(G/H)$, 
        \[ 
            d(r_1 \cdot r_2) = r_1 \cdot d(r_2) + d(r_1) \cdot r_2 ,
        \] 
        \item for all $H \subset K \subset G$, $r \in R(G/H)$, 
        \[ 
            d(\Nm_H^K r) = \Tr_H^K \Nm_{d_2} \Res_{d_1}(r) \cdot d(r) 
        \] 
        where $d_i$ is the restriction to the compliment of the diagonal of the projection onto the $i$th factor of $K/H \times K/H$, and
        \item $d$ vanishes on the image of $k$ in $R$.
    \end{enumerate}
    The set of genuine $k$-derivations from $R$ to $M$ is denoted $\Der_k(R,M)$.
\end{definition}

\begin{theorem}[{\cite[Theorem 5.7]{Hil17}}]
    There is a natural isomorphism 
    \[ 
        \Der_k(R,-) \cong \Hom_R(\Omega^{1,\Tamb}_{R/k},-) 
    \] 
    of functors.
\end{theorem}

\begin{definition}
	A morphism $k \rightarrow R$ of Tambara functors is \emph{\'{e}tale} if it is formally \'{e}tale and finitely presented.
\end{definition}

There is an analogous story for Green functors, with two main differences: First, in defining Green functor derivations, we only ask for conditions (1) and (3) in \cref{def:genuine-k-derivation} to hold (since condition (2) just does not make sense). Second, the K\"{a}hler differentials of a Green functor morphism $k \rightarrow R$ are constructed by setting
\[ 
    I := \mathrm{Ker}(R \boxtimes_k R \rightarrow R) 
\] 
and then defining
\[ 
    \Omega^{1,\Green}_{R/k} := I/I^2 . 
\]
With these definitions, $\Omega^{1,\Green}_{R/k}$ carries the universal (Green functor) $k$-derivation of $R$.

We record the following unpublished result of Mike Hill, Tyler Lawson, and Dylan Wilson. The proof was communicated to the author by Mike Hill and is included for the reader's convenience.

\begin{theorem}[Hill--Lawson--Wilson]\label{thm:Green-etale-implies-Tambara-etale}
    Let $k \rightarrow R$ be a morphism of Tambara functors. If $\Omega^{1,\Tamb}_{R/k} = 0$, then $\Omega^{1,\Green}_{R/k} = 0$. Consequently $k \rightarrow R$ is a formally \'{e}tale map of Tambara functors if and only if it is a formally \'{e}tale map of Green functors.
\end{theorem}

\begin{proof}
    Since $\Omega^{1,\Tamb}_{R/k}$ is a quotient of $\Omega^{1,\Green}_{R/k}$, we obtain a short exact sequence of $R$-modules 
    \[ 
        0 \rightarrow M \rightarrow \Omega^{1,\Green}_{R/k} \rightarrow \Omega^{1,\Tamb}_{R/k} \rightarrow 0
    \] 
    by defining $M$ to be the kernel. Since kernels and cokernels of $R$-modules are computed levelwise, it suffices to prove by induction on $|H|$ that each $M(G/H)$ is zero.

    In the base-case, note $\Omega^{1,\Green}_{R/k}(G/e)$ and $\Omega^{1,\Tamb}_{R/k}(G/e)$ are both naturally isomorphic to the ordinary K\"{a}hler differentials $\Omega^1_{R(G/e)/k(G/e)}$ (cf. the argument in \cref{lem:bottom-level-of-Kahler-differentials} below). Thus $M(G/e) = 0$.
    
    Now observe that $M(G/H)$ is generated by elements which are norms (using the Tambara structure on $R \boxtimes_k R$) of elements of $I(G/L)$, where $|L| < |H|$. By the inductive hypothesis that $\Omega^{1,\Green}_{R/k}(G/L) = 0$ for such $L$, any element of $I(G/L)$ is a sum over subconjugates $L'$ of $L$ of proper transfers of elements of $I(G/L')^2$. It suffices to observe each summand is zero in $M(G/H)$. 
    
    The exponential formula implies that the norm of any such sum is a sum of norms and proper transfers. The norm summands are zero in $M(G/H)$ because monoidality of the norm implies $\mathrm{Nm}_{L'}^H(I(G/L')^2) \subset I(G/H)^2$ (which is zero in $M(G/H) \subset \Omega^{1,\Green}_{R/k}(G/H)$). The proper transfer summands may be regarded as proper transfers in $M$, which are zero in $M(G/H)$ by our inductive hypothesis.
\end{proof}

%% file: Technical_Lemmas.tex
\section{Flat modules are free}

Let $\F$ be a field of characteristic $p$, $G = C_p$, and form the constant Tambara functor $\underline{\F}$. When $p = 2$ and $\F = \F_2$, \cite{DHM24} prove that a finitely generated $\underline{\F_2}$-module is flat if and only if it is free. This is a consequence of a structure theorem in \cite{DHM24} for $\underline{\F_2}$-modules which is known to fail at odd primes. In this section we show that this flatness result holds in greater generality.

\begin{lemma}\label{lem:coind-of-res-is-box-product}
    Let $k$ be a $G$-Green functor. The endofunctors 
    \[ 
        \Coind_H^G \Res_H^G : k \text{-} \Mod \rightarrow k \text{-} \Mod 
    \] 
    and 
    \[ 
        \Coind_H^G \Res_H^G k \boxtimes_k - : k \text{-} \Mod \rightarrow k \text{-} \Mod 
    \] 
    are naturally isomorphic.
\end{lemma}

\begin{proof}
    By \cite[Theorem F]{Wis25a} we can choose a $\Res_H^G k$-module $N$ naturally in $M$ so that we have a natural isomorphism 
    \[ 
        \Coind_H^G \Res_H^G k \boxtimes_k M \cong \Coind_H^G N 
    \] 
    of $\Coind_H^G \Res_H^G k$-modules (a fortiori of $k$-modules). It suffices to show $\Res_H^G M$ and $N$ are naturally isomorphic. 
    
    We apply the strong symmetric monoidal functor $\Res_H^G$, obtaining 
    \[ 
        \prod_{g \in H \backslash G / H} \Coind_{H \cap {}^g H}^H \Res_{H \cap {}^g H}^{{}^g H} {}^g (\Res_H^G k) \boxtimes_{\Res_H^G k} \Res_H^G M \cong \prod_{g \in H \backslash G / H} \Coind_{H \cap {}^g H}^H \Res_{H \cap {}^g H}^{{}^g H} {}^g N .
    \] 
    One checks straghtforwardly that this isomorphism restricts to an isomorphism 
    \[ 
       \Res_H^G k \boxtimes_{\Res_H^G k} \Res_H^G M \cong N 
    \] 
    of $\Res_H^G k$-modules on the identity double coset factor. This is the desired isomorphism.
\end{proof}

In fact, this admits a multiplicative refinment which will be useful later.

\begin{lemma}\label{lem:coind-adj-unit-is-base-changed}
    Let $k \rightarrow R$ a map of $G$-Green or $G$-Tambara functors. Then 
    \[ \begin{tikzcd} 
    k \arrow[r] \arrow[d] & R \arrow[d] \\
    \Coind_H^G \Res_H^G k \arrow[r] & \Coind_H^G \Res_H^G R 
    \end{tikzcd} \]
    is a pushout square of $k$-algebras.
\end{lemma}

\begin{proof}
    We must construct a natural isomorphism 
    \[ 
        \Coind_H^G \Res_H^G k \boxtimes_k R \cong \Coind_H^G \Res_H^G R 
    \] 
    of $k$-algebras. By \cite[Theorem F and Proposition 5.8]{Wis25a} it suffices to show that we obtain naturally isomorphic $H$-Tambara functors upon applying $\Res_H^G$. Now the claim follows from strong symmetric monoidality of $\Res_H^G$, the double coset formula for the restriction of a coinduction, and the fact that $- \boxtimes_k R$ commutes with finite products.
\end{proof}

It is not true that projective modules are flat in an arbitrary abelian category equipped with bilinear symmetric monoidal structure. Fortunately, this result is true in the world of equviariant algebra.

\begin{lemma}\label{lem:free-k-modules-are-flat}
    Let $k$ be a $G$-Green functor. All finitely generated projective $k$-modules are flat.
\end{lemma}

\begin{proof}
    Since sums and summands of flat modules are easily seen to be flat, it suffices to show that each $k \rightarrow \Coind_H^G \Res_H^G k$ is flat. Now we have natural isomorphisms 
    \[ 
        \Coind_H^G \Res_H^G k \boxtimes_k - \cong \Coind_H^G \Res_H^G (-) 
    \] 
    of endofunctors of $k$-modules. $\Res_H^G$ and $\Coind_H^G$ are exact (as they are each other's left and right adjoints) so the claim follows.
\end{proof}

Guided by the heuristic that flat modules are ``torsion-free" and viewing the kernel of restriction as torsion elements, we obtain the following.

\begin{lemma}\label{lem:flat-over-MRC-implies-MRC}
    Let $k$ be a $G$-Green functor such that all restrictions in $k$ are injective. If $M$ is a flat $k$-module, then all restrictions in $M$ are injective.
\end{lemma}

\begin{proof}
    Our assumption on $k$ implies that whenever $L \subset H$, the canonical $k$-algebra map 
    \[ 
        \Coind_H^G \Res_H^G k \rightarrow \Coind_L^G \Res_L^G k 
    \] 
    is injective. By \cref{lem:coind-of-res-is-box-product} and flatness of $M$ we deduce that 
    \[ 
        \Coind_H^G \Res_H^G M \rightarrow \Coind_L^G \Res_L^G M 
    \] 
    is injective. The composition 
    \[ 
        M(G/H) \cong (\Coind_H^G \Res_H^G M)(G/G) \rightarrow (\Coind_L^G \Res_L^G M)(G/G) \cong M(G/L) 
    \] 
    is the restriction $\Res_L^H$ in $M$. Since monics in $k \text{-} \Mod$ are detected levelwise, the claim follows.
\end{proof}

\begin{remark}\label{rem:bottom-level-of-modules-for-constant-F_p}
    Let $G = C_p$ and $\F$ a field of characteristic $p$. If $M$ is any $\underline{\F}$-module, then $M(C_p/e)$ is a $\F[C_p]$-module. Our characteristic assumption implies $\F[C_p] \cong \F[\sigma]/(\sigma-1)^p$, which is a principal ideal domain. If $M$ is finitely generated, then $M(C_p/e)$ is a finitely generated $\F[C_p]$-module, and consequently the structure theorem for finitely generated modules over principal ideal domains applies. In particular, it supplies an isomorphism
    \[ 
        M(C_p/e) \cong \oplus_i \F[\sigma]/(\sigma-1)^{a_i} 
    \] 
    where $i$ ranges through a finite indexing set and $1 \leq a_i \leq p$.
\end{remark}

We required \cref{thm:flat-implies-free-for-constant-F_p} to prove \cref{cor:finite-etale-group-schemes-modular-characteristic} in an earlier version of the article, although it is now no longer necessary. However, besides being of intrinsic interest, it may be useful in the study of finite \'{e}tale $\underline{\F}$-algebras when one drops the assumption that $\F$ is algebraically closed.

\begin{theorem}\label{thm:flat-implies-free-for-constant-F_p}
    Let $G = C_p$ and $\F$ a field of characteristic $p$. A finitely generated $\underline{\F}$-module is flat if and only if it is free.
\end{theorem}

\begin{proof}
    Free modules are flat by \cref{lem:free-k-modules-are-flat}. Conversely, let $M$ be a flat (finitely generated) $\underline{\F}$-module. Let $x \in M(C_p/e)$ be an element whose $C_p$-orbits are linearly independent; such an element exists if and only if in the direct sum decomposition of \cref{rem:bottom-level-of-modules-for-constant-F_p} there is a summand with $a_i = p$. Then $x$ generates a submodule of $M$ isomorphic to $\Coind_e^{C_p} \F$. This is an injective $\underline{\F}$-module (cf. \cite[Lemma 4.8]{CW25}) hence splits off as a summand. Since $\Coind_e^{C_p} \F$ is flat by the previous paragraph, $M$ is flat if and only if the complimentary summand is flat. Splitting off more such summands, we reduce to the case that no terms with $a_i = p$ appear in the direct sum decomposition of \cref{rem:bottom-level-of-modules-for-constant-F_p}.

    Next, recall from \cref{lem:flat-over-MRC-implies-MRC} that the restriction in $M$ is injective. The transfer is therefore determined by the sum over $C_p$-orbits. In $M(C_p/e)$, this corresponds to multiplication by 
    \[ 
        1+\sigma+\dots+\sigma^{p-1} = \frac{\sigma^p-1}{\sigma-1} = \frac{(\sigma-1)^p}{\sigma-1} = (\sigma-1)^{p-1} 
    \] 
    which is zero by the fact that we have reduced to the case 
    \[ 
        M(C_p/e) \cong \oplus_i \F[\sigma]/(\sigma-1)^{a_i} 
    \] 
    with $1 \leq a_i \leq p-1$. 

    Finally, let $D$ be the $\underline{\F}$-module specified by $D(C_p/C_p) = 0$ and $D(C_p/e) = \F$. There is an injection $D \rightarrow \underline{\F}$, and, since $M$ is flat, it follows that $D \boxtimes M \rightarrow M$ is injective. We compute 
    \[ 
        (D \boxtimes M)(C_p/C_p) \cong M(C_p/e)_{C_p}/\Res_e^{C_p}(M(C_p/C_p)) 
    \] 
    using the description of the box product in \cite{Maz13}, and the map to $M(C_p/C_p)$ is given by sending the class $[x]$ of $x \in M(C_p/e)$ to its transfer, which is zero. Since $D \boxtimes M \rightarrow M$ is injective, we deduce $M(C_p/e)_{C_p} / \Res_e^{C_p}(M(C_p/C_p)) \cong 0$. 

    Unwinding definitions, if follows that the $C_p$-orbits of the image of the restriction in $M$ generate $M(C_p/e)$. Since the restriction lands in the fixed points, the $C_p$-orbits of the image of the restriction are equal to the image of restriction. Thus the restriction in $M$ is surjective, hence an isomorphism. We have thus shown that $M$ is isomorphic to a sum of copies of $\underline{\F}$, which is free, as desired.
\end{proof}

\begin{corollary}\label{cor:flat-proj-and-free-are-equiv-for-constant-F_p}
    Let $G = C_p$ and $\F$ a field of characteristic $p$. The following three conditions on a finitely generated $\underline{\F}$-module are equivalent:
    \begin{enumerate}
        \item flat
        \item projective
        \item free.
    \end{enumerate}
\end{corollary}

\begin{proof}
    By \cref{thm:flat-implies-free-for-constant-F_p} flat and free are equivalent, and by \cite[Theorem 4.7]{CW25} projective and free are equivalent (since $\underline{\F}$ is a relatively finite dimensional Green meadow).
\end{proof}

\begin{remark}
    Morally, the most pathological behavior of field-like Tambara functors tends to be captured by $\underline{\F_p}$ when $G = C_p$, as these fail to be field-like Green functors. We therefore expect \cref{cor:flat-proj-and-free-are-equiv-for-constant-F_p} to be true for all field-like $G$-Tambara functors regardless of $G$.
\end{remark}

\section{Fundamental properties of \'{e}tale morphisms}

In this section we establish that compositions, products, and base-changes of \'{e}tale morphisms are \'{e}tale. We begin with compositions.

\begin{proposition}\label{prop:etale-preserved-under-comp}
    Let $f : k \rightarrow R$ and $g : R \rightarrow S$ be Tambara functor morphisms.
    \begin{enumerate}
        \item If $S$ is finitely presented over $R$ and $R$ is finitely presented over $k$, then $S$ is finitely presented over $k$.
        \item If $S$ is flat over $R$ and $R$ is flat over $k$, then $S$ is flat over $k$.
        \item If $\Omega^{1,\Tamb}_{S/R} = 0$ and $\Omega^{1,\Tamb}_{R/k} = 0$, then $\Omega^{1,\Tamb}_{S/k} = 0$.
    \end{enumerate}
    Consequently the class of \'{e}tale morphisms of Tambara functors is closed under composition.
\end{proposition}

\begin{proof}
    If $S$ is finitely presented over $R$ and $R$ is finitely presented over $k$, then $S$ is a finite colimit of free polynomial $R$-algebras $SR[x_{H_i}]$. Since base-change along $k \rightarrow k[x_{H_i}]$ preserves colimits and takes free algebras to free algebras, we deduce that each $R[x_{H_i}]$ is a finite colimit of free polynomial $k$-algebras. Thus $S$ is a finite colimit of free polynomial $k$-algebras.

    Second, we have a natural isomorphism 
    \[ 
        S \boxtimes_R R \boxtimes_k - \cong S \boxtimes_k - 
    \] 
    of functors, so that flatness of $S$ over $R$ and of $R$ over $k$ imply flatness of $S$ over $k$.
    
    Finally, we establish the claim about genuine K\"{a}hler differentials. Assume every genuine $k$-derivation of $R$ is zero and that every genuine $R$-derivation of $S$ is zero. Let $d : S \rightarrow M$ be a genuine $k$-derivation. Then $d \circ g$ is a genuine $k$ derivation of $R$, hence is zero. Thus $d$ is a genuine $R$-derivation of $S$, hence is zero.
\end{proof}

Now we may move on to studying products and \'{e}taleness.

\begin{lemma}\label{lem:genuine-differentials-of-product}
	Let $R_1$ and $R_2$ be $k$-algebras. Then we have 
    \[ 
        \Omega^{1, \Tamb}_{R_1 \times R_2/k} \cong \Omega^{1, \Tamb}_{R_1/R} \oplus \Omega^{1, \Tamb}_{R_2/R} 
    \] 
    as $R_1 \times R_2$-modules (where we view an $R_i$-module as an $R_1 \times R_2$-module by restriction along the projection).
\end{lemma}

\begin{proof}
	One straightforwardly checks that there is a natural isomorphism 
    \[ 
        \mathrm{Der}_k(R_1 \times R_2, M) \cong \mathrm{Der}_k(R_1,M_1) \oplus \mathrm{Der}_k(R_2,M_2) .
    \] 
    where $M \cong M_1 \oplus M_2$ is the isomorphism of \cite[Proposition 5.3]{Wis25a}. The result follows from Yoneda's lemma and \cite[Theorem 5.7]{Hil17}.
\end{proof}

\begin{proposition}\label{prop:prod-of-etale-is-etale}
    Let $R_1$ and $R_2$ be $k$-algebras.
    \begin{enumerate}
        \item $R_1 \times R_2$ is flat over $k$ if and only if $R_1$ and $R_2$ are flat over $k$.
        \item $\Omega^{1,\Tamb}_{R_1 \times R_2/k} = 0$ if and only if $\Omega^{1,\Tamb}_{R_1/k} = 0$ and $\Omega^{1,\Tamb}_{R_2/k} = 0$.
        \item $R_1 \times R_2$ is finitely presented over $k$ if and only if $R_1$ and $R_2$ are.
    \end{enumerate}
    Consequently $R_1$ and $R_2$ are \'{e}tale $k$-algebras if and only if $R_1 \times R_2$ is an \'{e}tale $k$-algebra.
\end{proposition}

\begin{proof}
    First, we note that products and direct sums are the same thing for modules. Since $\boxtimes_k$ is additive, the direct sum of flat $k$-modules is flat, and summands of flat $k$-modules are flat. Second, \cref{lem:genuine-differentials-of-product} implies that $\Omega^{1,\Tamb}_{R_1 \times R_2/k} \cong 0$ if and only if $\Omega^{1,\Tamb}_{R_1/k} \cong 0$ and $\Omega^{1,\Tamb}_{R_2/k} \cong 0$.

    Lastly, we establish the claim about finite presentability. If $R_1 \times R_2$ is a compact $k$-algebra, then $k \rightarrow R_1$ is the coequalizer of the two maps from $k[x_G]$ from the free polynomial $k$-algebra on a generator in level $G/G$ respectively classifying the choice of zero and the choice of the idempotent generating the kernel of the projection 
    \[ 
        \left( R_1 \times R_2 \right)(G/G) \cong R_1(G/G) \times R_2(G/G) \rightarrow R_1(G/G) . 
    \] 
    Therefore $R_1$ is a finite colimit of finitely presented $k$-algebras, hence is finitely presented. By symmetry $R_2$ is also finitely presented.
	 
	Conversely, suppose $R_1$ and $R_2$ are finitely presented $k$-algebras. In $(R_1 \times R_2)(G/e)$, let $x_1 = (1,0)$ and $x_2 = (0,1)$. If $S$ is any $k$-algebra, the set $\Hom_k(R_1 \times R_2, S)$ decomposes as the disjoint union 
    indexed by $G$-fixed idempotents $y \in S(G/e)$ of the set of $k$-algebra morphisms $R_1 \times R_2 \rightarrow S$ which send $x_1$ to $y$. Writing $S = yS_1 \times (1-y)S_2$ (using \cite[Proposition 3.6]{Wis25a}), we thus have
    \begin{equation}\label{eq:k-alg-product-hom-set} 
        \Hom_k(R_1 \times R_2, S) \cong \bigsqcup_y \Hom_k(R_1,yS) \sqcup \Hom_k(R_2,(1-y)S) . 
    \end{equation}
    
    Let $F : \mathsf{D} \rightarrow k \text{-} \Alg$ be a filtered diagram and $y$ a $G$-fixed idempotent of $\colim_{\mathsf{D}} F(d)$. Since filtered colimits are computed levelwise, and filtered colimits of rings commute with the choice of a finite list of elements satisfying a finite list of equations (by compactness of all quotients of $\Z[x_1,\dots,x_n]$ in the category of rings), by passing to a cofinal diagram, we may assume each $F(d)(G/e)$ contains a $G$-fixed idempotent $y_d$ mapping to $y$ in the colimit. By \cite[Proposition 3.6]{Wis25a} we are entitled to write $F(d) \cong y_d F(d) \times (1-y_d) F(d)$ as $k$-algebras. 
    
    In light of \cref{eq:k-alg-product-hom-set} and the fact that filtered colimits commute with all finite products and all coproducts, it suffices to show for each $y$ that we have an isomorphism 
    \begin{align*} 
        \Hom_k(R_1, y \cdot \colim_{\mathsf{D}} F(d)) & \sqcup \Hom_k(R_2,(1-y) \cdot \colim_{\mathsf{D}} F(d)) \\ 
        & \cong \colim_{\mathsf{D}} ( \Hom_k(R_1,y_d F(d)) \sqcup \Hom_k(R_1,(1-y_d)F(d)) ) .
    \end{align*}
    But this follows immediately from compactness of $R_1$ and $R_2$ over $k$.
\end{proof}

Next, we show that flat base-change preserves \'{e}taleness. We start by showing that genuine K\"{a}hler differentials enjoy the expected base-change property for flat morphisms.

\begin{proposition}\label{prop:Kahler-differentials-of-base-change}
    Let $f : k \rightarrow \ell$ be a flat map of Tambara functors, $R$ a $k$-algebra, and $S := R \boxtimes_k \ell$. Then we have a natural isomorphism 
    \[ 
        \Omega^{1,\Tamb}_{S/\ell} \cong \Omega^{1,\Tamb}_{R/k} \boxtimes_k \ell
    \]
    of $S$-modules.
\end{proposition}

\begin{proof}
    Let $I$ be the kernel of $R \boxtimes_k R \rightarrow R$ and $J$ the kernel of $S \boxtimes_\ell S \rightarrow S$. We have a commutative diagram 
    \[ \begin{tikzcd}
        (R \boxtimes_k  R) \boxtimes_k \ell \arrow[r, "\cong"] \arrow[d] & S \boxtimes_\ell S \arrow[d] \\
        R \boxtimes_k \ell \arrow[r, "\cong"] & S 
    \end{tikzcd} \]
    so that by flatness of $\ell$ over $k$, $I 
    \boxtimes_k \ell \cong J$. Note that this is an isomorphism of non-unital Tambara functors. It follows that $I^{>1} \boxtimes_k \ell^{>1} \cong J^{>1}$. Since $\ell$ is unital, $\ell^{>1} = \ell$, whence we compute 
    \[ 
        \Omega^{1,\Tamb}_{S/\ell} = J/J^{>1} \cong (I \boxtimes_k \ell)/(I^{>1} \boxtimes_k \ell) \cong (I/I^{>1}) \boxtimes_k \ell = \Omega^{1,\Tamb}_{R/k} \boxtimes_k \ell 
    \] 
    as desired.
\end{proof}

\begin{remark}
    When $G$ is the trivial group \cref{prop:Kahler-differentials-of-base-change} is known to hold without the flatness assumption. One way to remove the flatness assumption in the Tambara setting would be to mimic the proof of \cite[Lemma 10.131.12]{stacks-project}. The only difficulty one encounters is showing that, if $d : R \rightarrow M$ is a genuine $k$-differential, then $d \otimes_k 0 : S \rightarrow M \boxtimes_k \ell$ is a genuine $\ell$-differential. 
    
    Checking the first and third conditions in \cref{def:genuine-k-derivation} for $d \otimes_k 0$ to be a differential is straightforward, so that we obtain a Green functor analogue of \cref{prop:Kahler-differentials-of-base-change} without a flatness assumption below. However, checking the second condition of \cref{def:genuine-k-derivation} amounts to a very complicated-looking computation involving the exponential formula for norms. When $G = C_p$, the formula simplifies enough to check by hand, so that \cref{prop:Kahler-differentials-of-base-change} can be seen to hold without the flatness assumption.
\end{remark}

\begin{lemma}\label{lem:base-change-for-Green-Kahler-differentials}
    Let $k \rightarrow R$ and $k \rightarrow \ell$ be maps of Green functors. Then we have a natural isomorphism 
    \[ 
        \Omega^{1,\Green}_{R/k} \boxtimes_k \ell \cong \Omega^{1,\Green}_{R \boxtimes_k \ell / \ell}
    \] 
    of $R \boxtimes_k \ell$-modules.
\end{lemma}

\begin{proof}
    The argument of \cite[Lemma 10.131.12]{stacks-project} goes through, using the fact that $\Omega^{1,\Green}_{R/k}$ carries the universal Green functor derivation of $R$ over $k$.
\end{proof}

\begin{proposition}\label{prop:base-change-of-etale-is-etale}
    Let $k \rightarrow \ell$ be a map of $G$-Tambara functors. The base-change functor $- \boxtimes_k \ell$ enjoys the following properties:
    \begin{enumerate}
        \item if $R$ is finitely presented over $k$, then $R \boxtimes_k \ell$ is finitely presented over $\ell$.
        \item if $R$ is flat over $k$, then $R \boxtimes_k \ell$ is flat over $\ell$.
        \item if $\Omega^{1,\Tamb}_{R/k} = 0$, then $\Omega^{1,\Tamb}_{R \boxtimes_k \ell/\ell} = 0$.
    \end{enumerate}
    Hence base-change preserves \'{e}tale morphisms.
\end{proposition}

\begin{proof}
    Base-change preserves colimits and free polynomial algebras, hence preserves being finitely presented. If $R$ is flat over $k$, then 
    \[ 
        R \boxtimes_k \ell \boxtimes_\ell - \cong R \boxtimes_k - 
    \] 
    so that $R \boxtimes_k \ell$ is a flat $\ell$-module. 
    
    Finally, if $\Omega^{1,\Tamb}_{R/k} \cong 0$, then by \cref{thm:Green-etale-implies-Tambara-etale} $\Omega^{1,\Green}_{R/k} = 0$. \cref{lem:base-change-for-Green-Kahler-differentials} then implies that $\Omega^{1,\Green}_{R \boxtimes_k \ell / \ell} \cong 0$. Since $\Omega^{1,\Tamb}_{R \boxtimes_k \ell / \ell}$ is a quotient of $\Omega^{1,\Green}_{R \boxtimes_k \ell / \ell}$, we deduce that the former is zero, as desired.
\end{proof}

Finally, we study the interaction between coinduction and \'{e}taleness.

\begin{theorem}\label{thm:coind-are-etale}
	Let $k$ be any $G$-Tambara functor. Then the adjunction unit 
	\[ 
		k \rightarrow \Coind_H^G \Res_H^G k
	\] 
	is \'{e}tale.
\end{theorem}

\begin{proof}
    Our proof is by induction on $|G|$. If $H = G$, then there is nothing to prove, so we assume $|H| < |G|$.
    
    First, we show $\Coind_H^G \Res_H^G k$ is finitely presented by showing that it is compact. Let $F : \mathsf{D} \rightarrow k \text{-} \Alg$ be a filtered diagram. If $\colim_{\mathsf{D}} F(d)$ is not in the image of $\Coind_H^G$, then no $F(d)$ is in the image of $\Coind_H^G$, using the fact that $\colim_{\mathsf{D}} F(d)$ is an $F(d)$-algebra for each object $d$ of $\mathsf{D}$ and that fact that any Tambara functor receiving a map from a coinduced Tambara functor is coinduced (by \cite[Corollary G]{Wis25a}). Consequently 
    \[ 
        \colim_{\mathsf{D}} \Hom_k(\Coind_H^G \Res_H^G k, F(d)) \cong \colim_{\mathsf{D}} \emptyset \cong \emptyset \cong \Hom_k(\Coind_H^G \Res_H^G k, \colim_{\mathsf{D}} F(d)) .
    \]
    
    On the other hand, if $\colim_{\mathsf{D}} F(d)$ is in the image of $\Coind_H^G$, then since filtered colimits are computed levelwise, we see that for some $d$ in $\mathsf{D}$, the $G$-ring $F(d)(G/e)$ contains a type $H$-idempotent whose distinct $G$-orbits form a complete set of orthogonal idempotents (since filtered colimits commute with the existence of finitely many elements satisfying finite lists of equations by compactness of all quotients of $\Z[x_1,\dots,x_n]$ in the category of rings). Thus there is some $x$ in $\mathsf{D}$ such that $F(x) \cong \Coind_H^G R$, and by passing to a cofinal subdiagram we may assume by \cite[Corollary G]{Wis25a} that $F$ lands in $\Coind_H^G R$-algebras. 
    
    The map $k \rightarrow \Coind_H^G R$ is adjoint to $\Res_H^G k \rightarrow R$, so that applying $\Coind_H^G$ makes $\Coind_H^G R$ into a $\Coind_H^G \Res_H^G k$-algebra. Thus we may view $F$ as landing in $\Coind_H^G \Res_H^G k$-algebras. As filtered colimits are computed levelwise, changing the target category of $F$ in this way does not change the colimit. By \cite[Corollary G]{Wis25a}, $F$ factors through $\Coind_H^G$ in the sense that we may write $F = \Coind_H^G \circ E$ for some functor $E : \mathsf{D} \rightarrow \Res_H^G k \text{-} \Alg$.

    By inductive hypothesis, for each $g \in H \backslash G / H$,
    \[ 
        \Res_H^G k \rightarrow \Coind_{H \cap {}^g H}^H \Res_{H \cap {}^g H}^{{}^g H} {}^g \Res_H^G k 
    \] 
    is finitely presented. Thus, by \cref{prop:prod-of-etale-is-etale} and the double coset formula for the restriction of a coinduction, 
    \[ 
        \Res_H^G k \rightarrow \Res_H^G \Coind_H^G \Res_H^G k 
    \] 
    is finitely presented. Since $\Coind_H^G$ commutes with filtered colimits, we compute 
    \begin{align*}
        \Hom_k(\Coind_H^G \Res_H^G k, \colim_{\mathsf{D}} F(d)) & \cong \Hom_k (\Coind_H^G \Res_H^G k, \Coind_H^G \colim_{\mathsf{D}} E(d)) \\
        & \cong \Hom_{\Res_H^G k}(\Res_H^G \Coind_H^G \Res_H^G k, \colim_{\mathsf{D}} E(d)) \\
        & \cong \colim_{\mathsf{D}} \Hom_{\Res_H^G k}(\Res_H^G \Coind_H^G \Res_H^G k,E(d)) \\
        & \cong \colim_{\mathsf{D}} \Hom_{k}(\Coind_H^G \Res_H^G k,F(d)) . 
    \end{align*}
    
    We have thus shown that $\Coind_H^G \Res_H^G k$ is finitely presented over $k$. Next, by \cref{lem:free-k-modules-are-flat} free $k$-modules are flat. Finally, we check that the genuine K\"{a}hler differentials vanish.

    Let $I$ be the kernel of 
	\[ 
		\Coind_H^G \Res_H^G k \boxtimes_k \Coind_H^G \Res_H^G k \rightarrow \Coind_H^G \Res_H^G k . 
	\]
    Via \cref{lem:coind-adj-unit-is-base-changed}, the isomorphism
    \[ 
        \Coind_H^G \Res_H^G \Coind_H^G \Res_H^G k \cong \prod_{g \in H \backslash G / H} \Coind_{H \cap {}^g H}^G \Res_{H \cap {}^g H}^{{}^g H} {}^g \Res_H^G k , 
    \] 
    and \cite[Theorem A]{Wis25b}, $I$ is the coinduction of the ideal $J$ defined as the kernel of 
    \begin{align*}
        \prod_{g \in H \backslash G / H} \Coind_{H \cap {}^g H}^H \Res_{H \cap {}^g H}^{{}^g H} {}^g \Res_H^G k & \boxtimes_{\Res_H^G k} \prod_{g \in H \backslash G / H} \Coind_{H \cap {}^g H}^H \Res_{H \cap {}^g H}^{{}^g H} {}^g \Res_H^G k \\
        & \rightarrow \prod_{g \in H \backslash G / H} \Coind_{H \cap {}^g H}^H \Res_{H \cap {}^g H}^{{}^g H} {}^g \Res_H^G k .
    \end{align*}
    Additionally, one straightforwardly checks $I^{>1} = \Coind_H^G (J^{>1})$ (ultimately because $\res_H^G : G \text{-} \mathrm{Set} \rightarrow H \text{-} \mathrm{Set}$ preserves cardinality). Since coinduction preserves quotients we deduce 
	\begin{align*} 
		\Omega^{1,\Tamb}_{\Coind_H^G \Res_H^G k/k} & = I/I^{>1} \\
        & \cong \Coind_H^G (J/J^{>1}) \\
        & \cong \Coind \left( \Omega^{1,\Tamb}_{\prod_{g \in H \backslash G / H} \Coind_{H \cap {}^g H}^H \Res_{H \cap {}^g H}^{{}^g H} {}^g \Res_H^G k/\Res_H^G k} \right) \\ 
        & \cong \Coind_H^G 0 \cong 0
	\end{align*}
	where the second-to-last isomorphism follows from the third statement of \cref{prop:prod-of-etale-is-etale} and the inductive hypothesis.
\end{proof}

\begin{proposition}\label{prop:coind-of-Kahler-differentials}
	Let $k \rightarrow R$ be a morphism of $H$-Tambara functors. Then we have an isomorphism
	\[ 
		\Omega^{1,\Tamb}_{\Coind_H^G R / \Coind_H^G k} \cong \Coind_H^G \Omega^{1,\Tamb}_{R/k} 
	\]
	of $\Coind_H^G R$-modules.
\end{proposition}

\begin{proof}
    This follows from flat base-change along $k \rightarrow \Coind_H^G \Res_H^G k$ using \cref{lem:coind-adj-unit-is-base-changed,prop:Kahler-differentials-of-base-change,lem:coind-of-res-is-box-product}.
\end{proof}

\begin{proposition}\label{prop:coind-preserves-etale}
	Let $k \rightarrow R$ be a morphism of $H$-Tambara functors. 
    \begin{enumerate}
        \item If $R$ is finitely presented over $k$, then $\Coind_H^G R$ is finitely presented over $\Coind_H^G k$.
        \item $R$ is flat over $k$ if and only if $\Coind_H^G R$ is flat over $\Coind_H^G k$.
        \item $\Omega^{1,\Tamb}_{R/k} = 0$ if and only if $\Omega^{1,\Tamb}_{\Coind_H^G R/\Coind_H^G k} = 0$.
    \end{enumerate}
    Thus if $R$ is \'{e}tale over $k$, then $\Coind_H^G R$ is \'{e}tale over $\Coind_H^G k$. 
\end{proposition}

\begin{proof}
    The forwards direction is the special case of \cref{prop:base-change-of-etale-is-etale} applied to flat base-change along $k \rightarrow \Coind_H^G \Res_H^G k$, using \cref{lem:coind-adj-unit-is-base-changed}. The backwards direction on the flatness statement follows from the fact that $\Coind_H^G$ is an exact symmetric monoidal equivalence of abelian categories \cite[Theorem F]{Wis25a}, and the backwards direction on the statement on vanishing of the genuine K\"{a}hler differentials follows from \cref{prop:coind-of-Kahler-differentials} and the fact that $\Coind_H^G$ reflects zero.
\end{proof}

%% file: MainResults.tex
\section{Classification theorems}

We begin by showing that $\ev_{G/e}$ preserves finite affine \'{e}tale group schemes. Throughout this section we will sometimes implicitly use the identification $\Res_H^G \underline{R} \cong \underline{R}$ of constant $H$-Tambara functors at a ring $R$.

\begin{lemma}\label{lem:bottom-level-of-Kahler-differentials}
	Let $k \rightarrow R$ be a map of Tambara functors. We have a natural isomorphism
	\[ 
		\Omega^{1,\Tamb}_{R/k}(G/e) \cong \Omega^1_{R(G/e)/k(G/e)}
	\]
	of $k(G/e)$-modules.
\end{lemma}

\begin{proof}
	Let $I$ denote the kernel of $R \boxtimes_k R \rightarrow R$. In level $G/e$ this map has the form 
	\[ 
		R(G/e) \otimes_{k(G/e)} R(G/e) \rightarrow R(G/e) 
	\] 
	so it suffices to check $I^{>1}(G/e) = I(G/e)^2$. Since the only norms landing in the $G/e$-level of a Tambara functor are honest multiplications, the claim follows.
\end{proof}

\begin{proposition}\label{thm:bottom-level-preserves-finite-affine-etale-group-schemes}
	Fix a relatively finite dimensional $G$-Tambara functor $k$. The functor $\ev_{G/e}$ preserves the following:
	\begin{enumerate}
		\item finitely generated $k$-modules,
		\item flat $k$-modules,
		\item finitely generated $k$-algebras,
		\item finitely presented $k$-algebras, and
		\item cogroup structures over $k$.
	\end{enumerate}
	In particular, it preserves finite affine \'{e}tale group schemes. 
\end{proposition}

\begin{proof}
	$\ev_{G/e}$ preserves finitely generated modules by relative finite dimensionality of $k$. Since $\ev_{G/e}$ is strong symmetric monoidal with respect to the box product over $k$ and tensor product over $k(G/e)$, it takes flat $k$-modules to flat $k(G/e)$-modules.
	
	By \cite[Theorem A]{Bru05} all free $k$-algebras on finitely many generators are sent by $\ev_{G/e}$ to free $k(G/e)$-algebras on finitely many generators. Thus $\ev_{G/e}$ preserves finite generation. Since $\ev_{G/e}$ (viewed as landing in the category of rings) is left adjoint to $\Coind_e^G$, it preserves coequalizers, hence finite presentation. Finally, since $\ev_{G/e}$ is strong symmetric monoidal, it also preserves cogroup structures.
\end{proof}

\begin{corollary}\label{cor:ev-lands-in-etale-group-schemes-with-G-action}
    The functor 
    \[ 
        \ev_{G/e} : \fet_k \rightarrow \fet_{k(G/e)} 
    \]
    naturally factors through the forgetful functor 
    \[ 
        (\fet_{k(G/e)})^{BG} \rightarrow \fet_{k(G/e)} 
    \]
    from finite affine \'{e}tale group schemes with $G$-action to finite affine \'{e}tale group schemes.
\end{corollary}

\begin{proof}
    Let $R$ represent any finite affine \'{e}tale group scheme over $k$. The Weyl action defines a $G$-action on the ring $\ev_{G/e} R = R(G/e)$. The $G$-action is by cogroup homomorphisms as the $G/e$-level of the box product is the tensor product with the diagonal $G$-action, so it determines a $G$-action on the group scheme $R(G/e)$ represents.
\end{proof}

\begin{theorem}\label{thm:detecting-etale-at-bottom-level}
	Let $k$ be a $G$-Tambara functor and $\ell$ a flat $k$-algebra. Assume either
	\begin{enumerate}
		\item $\ell$ is cohomological and $|G|$ is invertible in $\ell(G/G)$, or
		\item all transfers in $\ell$ are surjective
	\end{enumerate}
	then $\ell$ is formally \'{e}tale over $k$ if and only if $\ell(G/e)$ is formally \'{e}tale over $k(G/e)$. If $\ell$ is finitely presented over $k$, then $\ell$ is \'{e}tale over $k$ if and only if $\ell(G/e)$ is \'{e}tale over $k(G/e)$.
\end{theorem}

\begin{proof}
	By \cref{lem:bottom-level-of-Kahler-differentials} we have 
    \[ 
        \Omega^{1,\Tamb}_{\ell/k}(G/e) \cong \Omega^1_{\ell(G/e)/k(G/e)} .
    \] 
    Since all restrictions in the $\ell$-module $\Omega^{1,\Tamb}_{\ell/k}$ are injective \cref{lem:sufficient-conditions-for-all-modules-to-be-FP}, $\Omega^{1,\Tamb}_{\ell/k} \cong 0$ if and only if $\Omega^{1,\Tamb}_{\ell/k}(G/e) \cong 0$.
\end{proof}

\begin{corollary}\label{cor:Galois-field-extension-gives-etale}
    Let $L$ be a $G$-Galois extension of a field $K$. Then 
    \[ 
        \underline{K} \rightarrow \FP(L) 
    \] 
    is formally \'{e}tale. If $\underline{K}$ satisfies the Hilbert basis theorem, then it is \'{e}tale.
\end{corollary}

\begin{proof}
    First, we observe $\FP(L)$ is flat. Observe the $K[G]$-module isomorphic $L \cong K[G] \cong (\Coind_e^G K)(G/e)$. Since $L$ and $\Coind_e^G K$ are fixed-point $\underline{K}$-modules, we have a chain of $k$-module isomorphisms 
    \[ 
        \FP(L) \cong \FP((\Coind_e^G K)(G/e)) \cong \Coind_e^G K . 
    \] 
    Thus $\FP(L)$ is a free $\underline{K}$-module, so that it is flat by \cref{lem:free-k-modules-are-flat}. Additionally, since all transfers in $\Coind_e^G K$ are surjective, we deduce that all transfers in $\FP(L)$ are also surjective, so that we are in the situation of \cref{thm:detecting-etale-at-bottom-level}.
    
    If $\underline{K}$ satisfies the Hilbert basis theorem, then since $\FP(L)$ is levelwise finite over $\underline{K}$, it is finitely generated over $\underline{K}$, hence finitely presented. Since Galois extensions of fields are \'{e}tale, \cref{thm:detecting-etale-at-bottom-level} applies.
\end{proof}

If $k$ is cohomological and $|G|$ is invertible in $k(G/G)$, then every $k$-algebra is fixed point by \cref{lem:sufficient-conditions-for-all-modules-to-be-FP}. Since fixed-point Mackey functors are cohomological, we deduce that every $k$-algebra $R$ is cohomological and that $|G|$ is invertible in $R(G/G)$. It would be interesting to have an application of \cref{thm:detecting-etale-at-bottom-level} in the case that $k$ is not cohomological and/or $|G|$ is not invertible in $k(G/G)$.

\begin{corollary}\label{cor:Theorem-Q}
	Let $R$ be a ring such that $|G|$ is a unit in $R$. If $\underline{R}$ satisfies the Hilbert basis theorem, then 
    \[ 
        \ev_{G/e} : \fet_{\underline{R}} \rightarrow (\fet_R)^{BG} 
    \] 
    is an equivalence of categories with inverse induced by $\FP$.
\end{corollary}

\begin{proof}
	Let $L$ be the representing ring of a finite affine \'{e}tale group scheme over $R$ with $G$-action. Then $L$ is a cogroup equipped with a $G$-action by cogroup automorphisms. Our assumption on $R$ implies that 
    \[ 
        \FP : R[G] \text{-} \Mod \rightarrow \underline{R} \text{-} \Mod 
    \] 
    is strong symmetric monoidal (it's the inverse equivalence to the strong symmetric monoidal functor $\ev_{G/e}$ by \cref{lem:sufficient-conditions-for-all-modules-to-be-FP}), so we deduce that $\FP(L)$ represents a finite affine group scheme over $\underline{R}$. By \cref{thm:detecting-etale-at-bottom-level}, $\FP(L)$ is \'{e}tale over $\underline{R}$. We have also used that each $L^H$ is a finite $R$-algebra, so that $\FP(L)$ is finitely generated over $\underline{R}$, hence finitely presented by the assumption that $\underline{R}$ satisfies the Hilbert basis theorem.
	
	It remains to check that $\ev_{G/e}$ and $\FP$ are mutually inverse, up to natural isomorphism. If $\ell$ represents an affine \'{e}tale group scheme over $k$, then the adjunction unit $\ell \rightarrow \FP(\ev_{G/e} \ell)$ is an isomorphism by \cref{lem:sufficient-conditions-for-all-modules-to-be-FP}. Conversely, if $L$ is the representing ring of an affine \'{e}tale group scheme over $k(G/e)$ with $G$-action, then $\ev_{G/e} \FP(L)$ clearly returns $L$.
\end{proof}

For example, \cref{cor:Theorem-Q} applies to the Tambara functor $\underline{R}$ whenever $\Q \subset R$, at least for the list of finite groups $G$ appearing in \cite[Theorem A]{Sun25}.

We now turn our attention to the case in which the characteristic possibly divides the order of $G$.

\begin{theorem}\label{thm:modular-characteristic-characterizing-finite-etale}
	Let $G$ be an arbitrary finite group and $\F$ any algebraically closed field. Then a finite $\underline{\F}$-algebra $\ell$ is \'{e}tale if and only if it is a finite product of $\underline{\F}$-algebras 
    \[ 
        \underline{\F} \rightarrow \Coind_H^G \underline{\F} . 
    \]
\end{theorem}

\begin{proof}
    By \cref{thm:bottom-level-preserves-finite-affine-etale-group-schemes} $\ell(G/e)$ is a finite \'{e}tale $\F$-algebra. Thus $\ell(G/e) \cong \prod_i \F$ for a finite indexing set; $G$ acts by permuting factors. Since $\ell(G/e)$ is Noetherian, \cite[Theorem D]{Wis25a} supplies a $G$-Tambara functor isomorphism 
    \[ 
		\ell \cong \prod_i \Coind_{H_i}^G \ell_i
	\] 
	where $\ell_i$ is an $H_i$-Tambara functor under $\underline{\F}$ such that $\ell_i(G/e) \cong \F$. By \cref{prop:prod-of-etale-is-etale} each $\ell_i$ is \'{e}tale, so without loss of generality we may take $\ell = \ell_i$.

    Now we induct on $|G|$. The base case is $G = \{ e \}$, wherein the result is classical. In the inductive step, we have two things to show: first, if $\ell \cong \Coind_H^G \ell'$, then $\ell'$ is an \'{e}tale $\Coind_H^G \underline{\F}$-algebra, so that by inductive hypothesis 
    \[ 
        \ell \cong \Coind_H^G \Coind_K^H \underline{\F} \cong \Coind_K^G \underline{\F} . 
    \] 
    Second, we must show that if $\ell(G/e) \cong \F$, then $\ell \cong \underline{\F}$

    For the first point, by \cref{lem:coind-adj-unit-is-base-changed} the base-change of $\underline{\F} \rightarrow \Coind_H^G \ell'$ along the flat map $\underline{\F} \rightarrow \Coind_H^G \underline{\F}$ is an \'{e}tale map 
    \[ 
        \Coind_H^G \underline{\F} \rightarrow \Coind_H^G \Res_H^G \Coind_H^G \ell' \cong \prod_{g \in H \backslash G / H} \Coind_{H \cap {}^g H}^G \Res_{H \cap {}^g H}^{{}^g H} {}^g \ell'
    \]
    By \cref{prop:prod-of-etale-is-etale} the projection onto the identity double coset factor describes an \'{e}tale map 
    \[ 
        \Coind_H^G \underline{\F} \rightarrow \Coind_H^G \ell' 
    \] 
    of $\underline{\F}$-algebras.
    
    For the second point, observe that all restrictions in $\ell$ are injective by \cref{lem:flat-over-MRC-implies-MRC}. For each $H \subset G$ we have a commutative diagram 
    \[ \begin{tikzcd} 
        \F \arrow[r, "="] & \underline{\F}(G/H) \arrow[r] \arrow[d, "\mathrm{Id}"] & \ell(G/H) \arrow[d, "\Res_e^H"] & \\
         & \underline{\F}(G/e) \arrow[r, "\cong"] & \ell(G/e) \arrow[r, "\cong"] & \F
    \end{tikzcd} \]
    from which it follows that $\underline{\F}(G/H) \rightarrow \ell(G/H)$ is an isomorphism. This establishes the claim.
\end{proof}

We are finally able to establish $G$-Galois descent for finite affine \'{e}tale group schemes along $\underline{\F} \rightarrow \Coind_e^G \F$ in the case of modular characteristic.

\begin{corollary}\label{cor:finite-etale-group-schemes-modular-characteristic}
    Let $\F$ be an algebraically closed field and $G$ an arbitrary finite group. Then
    \[ 
        \ev_{G/e} : \fet_{\underline{\F}} \rightarrow (\fet_\F)^{BG} 
    \] 
    is an equivalence of categories with inverse induced by $\FP$. The subcategory of commutative group schemes is equivalent to the category of finitely generated $\Z[G]$-modules.
\end{corollary}

\begin{proof}
    Clearly $\ev_{G/e} \circ \FP \cong \mathrm{Id}$. On the other hand, by \cref{thm:modular-characteristic-characterizing-finite-etale}, if $\ell$ represents a finite affine \'{e}tale group scheme, then the adjunction unit $\ell \rightarrow \FP(\ell(G/e))$ is an isomorphism. Thus $\FP \circ \ev_{G/e} \cong \mathrm{Id}$ as well.

    Note that commutative group schemes in $\fet_{\underline{\F}}$ correspond to functors $BG \rightarrow \fet_\F$ on the right which send the unique point of $BG$ to a commutative group scheme. The category of affine commutative \'{e}tale group schemes over $\F$ is equivalent to the category of finitely generated abelian groups, $\Z \text{-} \Mod^{\textrm{f.g.}}$ (this is essentially an immediate consequence of \cite[Proposition 2 on page A.V.30]{Bou90}). The final claim then follows from the equivalence 
    \[ 
        ( \Z \text{-} \Mod^{\textrm{f.g.}} )^{BG} \cong \Z[G] \text{-} \Mod^{f.g.} 
    \]
    of categories.
\end{proof}